\documentclass[11pt, letterpaper, oneside]{amsart}
\usepackage{tikz}
\usepackage{amsrefs}
\usepackage{amsthm}
\usepackage{mathtools}
\usepackage{genyoungtabtikz}
\usepackage{float}

\usetikzlibrary{decorations.pathreplacing,calligraphy}
\usetikzlibrary {decorations.pathmorphing, decorations.pathreplacing, decorations.shapes, }

\newtheorem{thm}{Theorem}[section]

\newtheorem{prop}[thm]{Proposition}
\theoremstyle{definition}

\newtheorem{exm}[thm]{Example}

\newtheorem{rem}[thm]{Remark}
\numberwithin{equation}{section}

\begin{document}

\title[Fixed hooks in arbitrary columns]{Fixed hooks in arbitrary columns}
\author{Philip Cuthbertson}
\address{Department of Mathematical Sciences\\ Michigan Technological University\\ Houghton, MI 49931} \email{pecuthbe@mtu.edu}

\begin{abstract}
    In a paper by the author, Hemmer, Hopkins, and Keith the concept of a fixed point in a sequence was applied to the sequence of first column hook lengths of a partition. In this paper we generalize this notion to fixed hook lengths in an arbitrary column of a partition. We establish combinatorial connections between these fixed hooks and colored partitions that have interesting gap and mex-like conditions. Additionally, we obtain several generating functions for hook lengths of a given fixedness by hook length or part size in unrestricted partitions as well as some classical restrictions such as odd and distinct partitions.
\end{abstract}

\maketitle

\section{Introduction}

    Let $\lambda = (\lambda_1, \lambda_2, \ldots, \lambda_r)$ be a weakly decreasing sequence of non-negative integers. $\lambda$ is called a \textit{partition} of $n$ if $\lambda_1 + \lambda_2 + \ldots + \lambda_r = n$ and then each $\lambda_i$ is called a part. The \textit{Young diagram} of a partition is a representation of the partition in which each part is represented by a row of $\lambda_i$ many squares justified to the top right. Let $\lambda=(\lambda_1, \lambda_2, \ldots, \lambda_k)$ be a partition of $n$ and $\lambda'$ denote its conjugate. For a square $(i, j)$ in the Young diagram of $\lambda$, define $h_{i,j}(\lambda) = \lambda_i + \lambda'_j - i - j + 1$ to be the \textit{hook length of (i, j)}. In terms of the Young diagram, this is sum of the number of squares in the $i$th row and column at least $j$, called the \textit{arm}, and the number of squares in the $j$th column and row at least $i$, called the \textit{leg}.

    In \cite{BlecherFixedPointsRamanujan} Blecher and Knopfmacher defined the notion of a fixed point in a partition so that the partition $\lambda$ has a \textit{fixed point} if there is some $i$ such that $\lambda_i = i$. This was extended by Hopkins and Sellers \cite{hopkins2023blecher} so that $\lambda$ is said to have an \textit{h-fixed point} if $\lambda_i = i + h$ for some $i$. Instead of looking at the sequence of parts, the author, Hemmer, Hopkins, and Keith in \cite{CHHK} considered instead $\{h_{1, 1}(\lambda), h_{2, 1}(\lambda), \ldots, h_{t, 1}(\lambda)\}$, the sequence of first column hook lengths. This sequence uniquely defines the partition $\lambda$ and is notable in representation theory. Here we concern ourselves with the sequence of $m$th column hook lengths $\{h_{1, m}(\lambda), h_{2, m}(\lambda), \ldots, h_{t, m}(\lambda)\}$. Note this only uniquely defines the parts of the partition $\lambda$ of size at least $m$.
    \begin{figure}
        \centering
        \young(7542,6431,421,1)
        \caption{Young diagram of $(4, 4, 3, 1)$ labeled by hook lengths}
        \label{fig:YoungExample}
    \end{figure}
    
    We also define the $q-$Pochhammer symbol along with the Gaussian binomial coefficient and make use of their standard notation and combinatorial interpretations (see, for instance, Andrews \cite{AndrewsBook}):
    \begin{align*}
        (a;b)_n & := \prod_{k=0}^{n-1}(1-ab^k)\\
        (a;b)_\infty & := \prod_{k=0}^{\infty}(1-ab^k)\\
        \binom{a}{b}_q & := \frac{(q;q)_a}{(q;q)_b(q;q)_{a-b}}
    \end{align*}

    We now generalize some of the combinatorial theorems from \cite{CHHK} Noting that in each case their original theorems can be obtained by setting $m = 1$. Their Theorem 2.1 is generalized in the following theorem.

    \begin{thm}\label{ColorPartition}
        The number of partitions of $n$ having a 0-fixed hook in the $m$th column is equal to the sum over $L$ of the number of times across all partitions of $n$, with two colors of parts $1, 2, \ldots, m - 1$, that a part of size $L$ appears exactly $L + m - 1$ times in the first color, but $L+1, L+2, \ldots, L + 2m - 2$ are not parts in the first color.
    \end{thm}
    
    \begin{exm} Using $m = 3$ and looking at partitions of 10 we find the two sets,
    \begin{table}[h]
        \centering
        \begin{tabular}{|c|c|}
            \hline
            Description 1 & Description 2\\
            \hline
            $(6, 4)$ & $(7, 1^3)$ \\
            $(5, 4, 1)$ & $(6, \textcolor{red}{1}, 1^3)$ \\
            $(4, 4, 2)$ & $(\textcolor{red}{2}, 2^4)$ \\
            $(4, 4, 1, 1)$ & $(2^4, \textcolor{red}{1}^2)$ \\
            $(4, 3, 3)$ & $(2^4, \textcolor{red}{1}, 1)$ \\
            $(3, 3, 3, 1)$ & $(2^4, 1^2)$ \\
            $(3, 2, 2, 2, 1)$ & $(\textcolor{red}{2}^3, \textcolor{red}{1}^1, 1^3)$ \\
            $(3, 2, 2, 1, 1, 1)$ & $(\textcolor{red}{2}^2, \textcolor{red}{1}^3, 1^3)$ \\
            $(3, 2, 1, 1, 1, 1, 1)$ & $(\textcolor{red}{2}, \textcolor{red}{1}^5, 1^3)$ \\
            $(3, 1, 1, 1, 1, 1, 1, 1)$ & $(\textcolor{red}{1}^7, 1^3)$ \\
            \hline
        \end{tabular}
    \end{table}
    \end{exm}

    Generalizing Theorems 3.3, 3.4, and 4.3 from \cite{CHHK} we have the following theorems

    \begin{thm}
    The number of partitions of n with an h-fixed hook in the mth column arising from a part of size m equals the number of times in all partitions of $n-mh$ where $m$ appears as a part exactly once, there are no parts of sizes $m+1, m+2, \ldots, 2m-1$, and there are at least $-h$ parts of size at least $2m$.
    \end{thm}

    Generalizing Theorem 3.4 from \cite{CHHK} we get the following theorem.
    
    \begin{thm}
    The number of times in all partitions of $n$ that an $h$-fixed hook arises from a part of size $k$ in the $m$th column equals the number of partitions of $n + \binom{k - m + 1}{2} - k(k - h - m + 1))$ where there are two colors of parts $1, 2, \ldots, m - 1$, parts of size $k - m + 1, \ldots, k + m - 1$ do not appear in the first color, and at least one part of every size $1, 2, \ldots, k - m$ appear in the first color.
    \end{thm}

    Generalizing Theorem 4.3 from \cite{CHHK} we have the following theorems.

    \begin{thm}
        The generating function for the number of $m$th column hooks of size $k$ in all partitions of $n$ is 
        \[\frac{q^{km}}{(q^k; q)_\infty} \sum_{l = 1}^k \frac{q^{-(l - 1)(m - 1)}(q^m; q)_{l - 1}}{(q; q)_{l - 1}(q; q)_{k - l}}\]
    \end{thm}

    In order to prove the previous theorems, we first establish several generating functions that are also of independent interest.

\section{Part Sizes}
    
    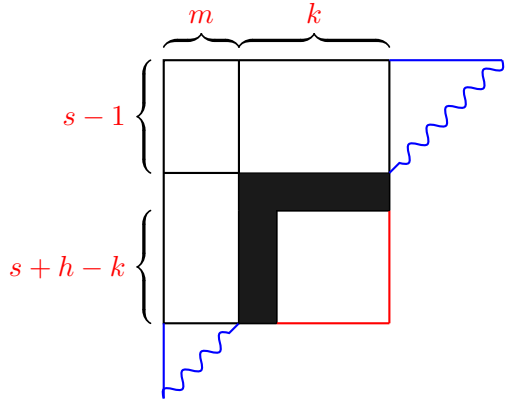
\begin{figure}[H]

\centering

   \begin{tikzpicture}[scale=0.5]
   
        \draw[thick] (0,-3) -- (-2,-3) -- (-2, 4) -- (0, 4) ;

        \draw[decorate,decoration={coil,aspect=0}, blue, thick] (-2,-5)--(0,-3);
        \draw[blue][thick] (-2,-5) -- (-2,-3);
        
        \draw [thick] (0,0) -- (0,4);
        \draw [thick] (0,4) -- (4,4);
        \draw [thick] (4,4) -- (4,0);
        \draw [thick] (4,0) -- (0,0);
        \draw[thick] (0, 1) -- (-2,1);
        \draw[thick] (0,0)--(0,-3);
        \draw[thick] (0,-3)--(1,-3);
        \draw[thick] (1,0)--(1,-3);   
        \draw[red][thick] (1,-3)--(4,-3);
        \draw[red][thick] (4,0)--(4,-3);
        \draw[blue][thick] (4,4)--(7,4);

        \filldraw[fill=black!90!white, draw=black] (0, 1) -- (4, 1) -- (4, 0) -- (1, 0) -- (1, -3) -- (0, -3) -- (0, 1);
             
       \draw[decorate,decoration={coil,aspect=0}, blue, thick]   (7,4)--(4,1);

    \draw [red,
    decorate,
    decoration = {calligraphic brace,
        raise=5pt,
        amplitude=5pt,
        aspect=0.5},line width=1.25pt] (-2,1) --  (-2,4)
node[pos=0.5,left=10pt,red, scale=1]{$s-1$};

   \draw [red,
    decorate,
    decoration = {calligraphic brace,
        raise=5pt,
        amplitude=5pt,
        aspect=0.5},line width=1.25pt] (0,4) --  (4,4)
node[pos=0.5,above=10pt,red, scale=1]{$k$};

\draw [red,
    decorate,
    decoration = {calligraphic brace,
        raise=5pt,
        amplitude=5pt,
        aspect=0.5},line width=1.25pt] (-2,4) --  (0,4)
node[pos=0.5,above=10pt,red, scale=1]{$m$};
 
       \draw [blue,
    decorate,
    decoration = {calligraphic brace,
        raise=5pt,
        amplitude=5pt,
        aspect=0.5},line width=1.25pt] (-2,-3) --  (-2,0)
node[pos=0.5,left=10pt,red, scale=1]{$s+h-k$};

    \end{tikzpicture}

\caption{$\lambda$ with an $h$-fixed hook $h_{m,s}=s+h$ at part $\lambda_s=k$.}
  \label{fig:lambadforGF}

\end{figure}
    
    The above figure will be useful for constructing the generating functions in the next two sections. The fixed hook is marked in black and we use the diction "below the fixed hook" to refer to the red box in the diagram. Theorem 3.1 of \cite{CHHK} establishes a generating function that counts fixed points in the sequence of first column hook lengths in any partition of $n$. We include this theorem for completeness.
    
    \begin{thm}\label{FixedbyPart}
    The generating function for the number of partitions of $n$ with an $h$-fixed hook arising from a part of size $k$ is
    \[\sum_{s = k - h}^\infty\frac{q^{(k + 1)(s - 1) + h + 1}}{(q; q)_{s - 1}} \binom{s + h - 1}{k - 1}_q = \sum_{s = 0}^\infty\frac{q^{s(k + 1) + k(k - h)}}{(q; q)_{s + k - h - 1}} \binom{s + k- 1}{k - 1}_q\]
    \end{thm}
    
    We generalize the above to an arbitrary column in the following theorem.
    
    \begin{thm}\label{mFixedbyPart}
    The generating function for the number of partitions of $n$ with an $h$-fixed hook in the $m$th column arising from a part of size $k \geq m$ is given by, 
    \[\sum_{s = k - h - m + 1}^{\infty} \frac{q^{s(k + m) + m(h - k + m - 1)}}{(q;q)_{s - 1} (q;q)_{m - 1}} \binom{s + h - 1}{k - m}_q\]
    \[= \sum_{s = 0}^{\infty} \frac{q^{s(k + m) + k(k - h - m + 1)}}{(q;q)_{s + k - h - m} (q;q)_{m - 1}} \binom{s + k - m}{k - m}_q\]
    \end{thm}

    \begin{proof}
    Using Theorem \ref{FixedbyPart} we extend the Young diagram by appending $m - 1$ columns to the left. Each column must be at least $2s + h - k$ long without any further restrictions. This is generated by $q^{(m - 1) (2s + h - k)}/(q; q)_{m - 1}$ giving
    \[\sum_{s = k - h}^\infty\frac{q^{(k + 1)(s - 1) + h + 1 + (m - 1) (2s + h - k)}}{(q;q)_{s - 1} (q; q)_{m - 1}} \binom{s + h - 1}{k - 1}_q\]
    Parts of size $k$ are now of size $k + m - 1$. After reindexing $k + m - 1 \to k$ one gets the theorem. The second line follows by reindexing $s \to s + k - h - m + 1$.
    \end{proof}

    One can derive similar formulae for certain restrictions on partitions. For instance, when considering just those partitions in which each part is odd, one arrives at the following theorem.
    
    \begin{thm}\label{OddbySize}
    The generating function for the number of odd partitions of $n$ with an $h$-fixed hook in the $m$th column arising from a part of size $k \geq m$ is given by, 
    \[\sum_{s = k - h - m + 1}^{\infty} \frac{q^{s(k + m) + m (m + h - k - 1)}}{(q^2; q^2)_{s - 1} (q; q^2)_{(m - 1)/2}} \binom{s + h - k + (k - m)/2}{s + h - k}_{q^2}\]
    if m is odd and
    \[\sum_{s = k - h - m + 1}^{\infty} \frac{q^{s (k + m + 1)+m (h - k + m) + h - k + 1}}{(q^2; q^2)_{s - 1} (q; q^2)_{m/2}} \binom{s + h - k + (k - m - 1)/2}{s + h - k}_{q^2}\]
    if m is even.
    \end{thm}

    \begin{proof}
    In the case that $m=1$, one gets the formula
    \[\sum_{s = k - h}^{\infty} \frac{q^{k (s - 1) + h + s}}{(q^2; q^2)_{s - 1}} \binom{s + h - k + (k - 1)/2}{s + h - k}_{q^2}\]
    The term of $q^{k(s - 1)}/(q^2; q^2)_{s - 1}$ generates the $s - 1$ rows above the part of size $k$ we are considering. The term of $q^{h + s} = q^{k + (s + h - k)}$ times the binomial coefficient generates the fixed hook and the $s + h - k$ rows below it. Generalizing to an arbitrary odd $m$ we add $m - 1$ columns to the left of the partition. Each column must be at least $2s + h - k$ long and still into odd parts which is generated by $q^{(m - 1)(2s + h - k)}/(q; q^2)_{m - 2}$ which gives
    \[\sum_{s = k - h}^{\infty} \frac{q^{(m - 1)(2s + h - k) + k (s - 1) + h + s}}{(q^2; q^2)_{s - 1}(q; q^2)_{(m - 1)/2}} \binom{s + h - k + (k - 1)/2}{s + h - k}_{q^2}\]
    Parts of size $k$ are now of size $k + m - 1$. After reindexing $k + m - 1 \to k$ and simplifying the exponent, we have the first equation. The case that $m$ is even is similar, but with two minor changes. We need all parts to be odd, so we insert a $q^{s + h - k}$ under the hook and change the binomial coefficient to $\binom{s + h - k + (k - 2)/2}{s + h - k}_{q^2}$. We also change the second Pochhammer to $(q; q^2)_{m/2}$. Which gives
    \[\sum_{s = k - h}^{\infty} \frac{q^{(m - 1)(2s + h - k) + k (s - 1) + h + s + (s + h - k)}}{(q^2; q^2)_{s - 1}(q; q^2)_{m/2}} \binom{s + h - k + (k - 2)/2}{s + h - k}_{q^2}\]
    After reindexing $k + m - 1 \to k$ and simplifying the exponent, we have the second equation.
    \end{proof}
    
    \begin{thm}\label{DistinctbySize}
    The generating function for the number of distinct partitions of $n$ with an $h$-fixed hook in the $m$th column arising from a part of size $k \geq m$ is given by,
    \[\sum_{s = 0}^{k - m} \frac{q^{s(k + m) + k(k - h - m + 1) + \binom{s + k - m + 1 - h}{2} + \binom{s}{2}}(-q; q)_{m - 1}}{(q; q)_{k - h - 1}} \binom{k - m}{s}_q\]   
    \end{thm}

    \begin{proof} 
        In the case that m = 1, one gets the formula
        \[\sum_{s = k - h}^{2k - h - 1} \frac{q^{(s-1)k + s + h + \binom{s}{2} + \binom{s + h - k}{2}}}{(q; q)_{s - 1}} \binom{k - 1}{s + h - k}_q\]
        Each of the $s - 1$ rows above the part of size $k$ we are considering are at least $k + 1$ giving the term of $q^{(s - 1)k}$ and the term $q^{\binom{s}{2}}/(q; q)_{s - 1} = q^{1 + 2 + \ldots + s - 1}/(q; q)_{s - 1}$ guarantees these parts are all distinct. The hook itself is generated by $q^{s + h}$. The portion of the partition under the hook is generated by $q^{\binom{s + h - k}{2}} \binom{k - 1}{s + h - k}_q$ the finite bounds on the sum are due to the limited possible part sizes below a part of size $k$ in a distinct partition. Generalizing to an arbitrary $m$ we add $m - 1$ columns to the left of the partition. Each column must be at least $2s + h - k$ long and still into distinct parts which is generated by $q^{(m - 1)(2s + h - k)}(-q; q)_{m - 1}$ which gives
        \[\sum_{s = k - h}^{2k - h - 1} \frac{q^{(s-1)k + s + h + \binom{s}{2} + \binom{s + h - k}{2} + (m - 1)(2s + h - k)}(-q; q)_{m - 1}}{(q; q)_{s - 1}} \binom{k - 1}{s + h - k}_q\]
        Parts of size $k$ are now of size $k + m - 1$. After reindexing $k + m - 1 \to k$ and simplifying the exponent, we have 
        \[\sum_{s = k - m + 1 - h}^{2k - 2m + 1 - h} \frac{q^{s(k + m) + m(m + h - k - 1) + \binom{s}{2} + \binom{s + h - k + m - 1}{2}}(-q; q)_{m - 1}}{(q; q)_{k - h - 1}} \binom{k - m}{s + h - k + m - 1}_q\]
        Lastly, reindexing $s \to s + k - m + 1 - h$ gives the theorem.
    \end{proof}

\section{Hook Sizes}

    The generating functions tend to be simpler when considering fixed hooks arising from a given hook size instead of from a given part size. Thus, this section will focus primarily on this interpretation. Previously established \cite{CHHK} was a generating function analogous to Theorem \ref{FixedbyPart} but by a given hook length. Another benefit of counting by hook length is that the generating functions can be summed over all $h$ and $m$ to get the generating function for the number of hooks of a given length in certain restricted families of partitions. We restate Theorem 4.1 from \cite{CHHK} for completeness.
    
    \begin{thm}\label{FixedbyHook} The generating function for the number of partitions of $n$ with an $h$-fixed hook in the first columm that arises from a hook of size $k$ is
    \[\sum_{l = 1}^{k} \frac{q^{k + l(k - h - 1)}}{(q; q)_{k - h - 1}} \binom{k - 1}{l - 1}_q\]
    \end{thm}
    
    We then generalize this to an arbitrary column in the following theorem.
    
    \begin{thm}\label{mFixedbyHook} The generating function for the number of partitions of $n$ with an $h$-fixed hook in the $m$th column that arises from a hook of size $k$ is
    \[\sum_{l = 1}^{k} \frac{q^{(m - 1) (2 k - h - l) + k + l(k - h - 1)}}{(q; q)_{k - h - 1}(q; q)_{m - 1}} \binom{k - 1}{l - 1}_q\]
    \end{thm}
    
    \begin{proof}
    Using Theorem \ref{FixedbyHook} we extend the Ferrers diagram by adding in $m-1$ columns to the left. Each column must be at least $k - h - 1 + k - l + 1 = 2k - h - l$ long without any further restrictions. This is generated by $q^{(m - 1) (2 k - h - l)}/(q; q)_{m - 1}$ which gives the theorem.
    \end{proof}
    
    Similar to the previous section, we specialize to obtain similar generating functions for certain families of partitions.

    \begin{thm}\label{OddbyHook} The generating function for the number of odd partitions of $n$ with an $h$-fixed hook in the $m$th column that arises from a hook of size $k$ is
    \[\sum_{\substack{l = 1 \\ l \text{ odd}}}^{k} \frac{q^{k + l(k - h - 1) + (m - 1)(2k - h - l)}}{(q^2; q^2)_{k - h - 1} (q; q^2)_{(m - 1)/2}} \binom{k - l + (l - 1)/2}{k-l}_{q^2}\]
    if m is odd and
    \[\sum_{\substack{l = 1 \\ l \text{ even}}}^{k} \frac{q^{k + l(k - h - 1) + (m - 1)(2k - h - l) + (k - l)}}{(q^2; q^2)_{k - h - 1} (q; q^2)_{m/2}} \binom{k - l + (l - 2)/2}{k-l}_{q^2}\]
    if m is even.
    \end{thm}

    \begin{proof}
    Note that in all cases we need the arm length, $l$, and $m$ to have the same parity in order to guarantee that the fixed hook is arising from an odd part. In the case that $m = 1$, one gets the formula
    \[\sum_{\substack{l = 1 \\ l \text{ odd}}}^{k} \frac{q^{k + l(k - h - 1)}}{(q^2; q^2)_{k - h - 1}} \binom{k - l + (l - 1)/2}{k - l}_{q^2}\]
    Where the sum is taken over possible arm lengths of the hook. The term of $q^{l(k - h - 1)}/(q^2; q^2)_{k - h - 1}$ generates the $k - h - 1$ rows above the hook of size $k$ we are considering. The term of $q^{k}$ generates the fixed hook and the binomial coefficient generates the portion of the partition under the hook. Generalizing to an arbitrary odd $m$ we add $m - 1$ columns to the left of the partition. Each column must be at least $2k - h - l$ long and still into odd parts which is generated by $q^{(m - 1)(2k - h - l)}/(q; q^2)_{(m - 1)/2}$ which gives the first equation. The case that $m$ is even is similar. We need all parts to be odd, so we insert a $q^{k - l}$ under the hook and change the binomial coefficient to $\binom{k - l + (l - 2)/2}{k - l}_{q^2}$. We also change the second Pochhammer to $(q; q^2)_{m/2}$. Summing over even $l$ instead of odd gives the second equation.
    \end{proof}
    
    \begin{thm}\label{DistinctbyHook}
    The generating function for the number of distinct partitions of $n$ with an $h$-fixed hook in the $m$th column that arises from a hook of size $k$ is
    \[\sum_{l = \lceil (k + 1) / 2 \rceil}^{k} \frac{q^{k + l (k - h - 1) + (m - 1)(2k - h - l) + \binom{k - h}{2} + \binom{k - l}{2}}(-q; q)_{m - 1}}{(q; q)_{k - h - 1}} \binom{l - 1}{k - l}_q\]
    \end{thm}

    \begin{proof}
    In the case that $m = 1$, one gets the formula
    \[\sum_{l = \lceil (k + 1) / 2 \rceil}^{k} \frac{q^{k + l(k - h - 1) + \binom{k - h}{2} + \binom{k - l}{2}}}{(q; q)_{k - h - 1}} \binom{l - 1}{k - l}_q\]
    The term $q^{k + l(k - h - 1) + \binom{k - h}{2}}/(q; q)_{k - h - 1}$ generates the hook and the $k - h - 1$ distinct rows above the hook. The term $q^{\binom{k - l}{2}}\binom{l - 1}{k - l}_q$ generates the distinct parts below the hook. Generalizing to an arbitrary $m$ we add $m - 1$ columns to the left of the partition. Each column is at least $2k - h - l$ long and still distinct which is generated by $q^{(m - 1)(2k - h - l)}(-q; q)_{m - 1}$ giving the theorem.
    \end{proof}

    \begin{rem}
    For a fixed $k$, summing the generating functions from Theorems \ref{OddbyHook} and \ref{DistinctbyHook} over all $m$ and $h$ gives the generating function for the number of hooks of size $k$ in all odd partitions of $n$ and the number of hooks of size $k$ in all distinct partitions of $n$. These two functions generate $a_k(n)$ and $b_k(n)$ as defined in \cite{BBCFW} which is equivalent to Theorem 2.1 in \cite{CDH}.
    \end{rem}
    
    \begin{thm}\label{OddDistinctbyHook}
    The generating function for the number of odd and distinct partitions of $n$ with an $h$-fixed hook in the $m$th column that arising from a hook of size $k$ is
    \[\sum_{\substack{l = \lceil\frac{2k-1}{3}\rceil \\ l \text{ odd}}}^{k} \frac{q^{(m - 1) (2 k - h - l) + k + l(k - h - 1) + 2 \binom{k - h}{2} + 2 \binom{k - l}{2}}(-q, q^2)_{\frac{m - 1}{2}}}{(q^2; q^2)_{k - h - 1}}\binom{k - l + \frac{3l - 2k}{2}}{k - l}_{q^2}\]
    if $m$ is odd and
    \[\sum_{\substack{l = \lceil\frac{2k}{3}\rceil \\ l \text{ even}}}^{k} \frac{q^{k - l + (m - 1) (2 k - h - l) + k + l(k - h - 1) + 2 \binom{k - h}{2} + 2 \binom{k - l}{2}}(-q, q^2)_{\frac{m}{2}}}{(q^2; q^2)_{k - h - 1}}\binom{k - l + \frac{3l - 2k - 2}{2}}{k - l}_{q^2}\]
    if $m$ is even.
    \end{thm}

    \begin{proof}
    Note that $l$ and $m$ must have the same parity since $m - 1 + l$ must be an odd part size. Considering first $m$ odd, we have the term $q^{(m - 1)(2 k - h - l)}(-q, q^2)_{(m - 1)/2}$ generating the first $m - 1$ columns. The term of $q^{l(k - h - 1) + 2 \binom{k - h}{2}}/(q^2; q^2)_{k - h - 1} = q^{l(k - h - 1) + 2 + 4 + \ldots + 2(k - h - 1)}/(q^2; q^2)_{k - h - 1}$ generates the rest of the first $k - h - 1$ rows of the partition. The last factor is $q^{2\binom{k - l}{2}}\binom{k - l + \frac{3l - 2k}{2}}{k - l}_{q^2}=q^{2 + 4 + \ldots + 2(k - l - 1)}\binom{k - l + \frac{3l - 2k}{2}}{k - l}_{q^2}$ which generates the portion under the hook while guaranteeing each part is still distinct and odd. The case that $m$ is even is almost identical. There is an extra term of $q^{k - l}$ that is inserted under the hook adding 1 to each part to make them all odd. The Pochhammer is now $(-q, q^2)_{m/2}$ and the binomial coefficient is now $\binom{k - l + \frac{3l - 2k - 2}{2}}{k - l}_{q^2}$.
    \end{proof}

    Summing \ref{OddDistinctbyHook} over all $m$ and all $h$ gives a way to count all hooks of a given length in all odd and distinct partitions of $n$. While not the cleanest representation the following proposition does give an explicit generating function for the number of hooks of a given length in all odd and distinct partitions.

    \begin{prop}\label{OddDistinctTotal}
    The generating function for the number of hooks of length $k$ in all odd and distinct partitions of $n$ is given by:
    \[(-q; q^2)_{\infty}q^k\sum_{\substack{l = \lceil\frac{2k-1}{3}\rceil \\ l \text{ odd}}}^{k} q^{2\binom{k - l}{2}}\binom{k - l + \frac{3l - 2k}{2}}{k - l}_{q^2}\sum_{m = 0}^{\infty}\frac{q^{2m(k - l + 1)}}{(-q^{2m + 1}; q^2)_{\frac{l - 1}{2}}}\]
    \[+ (-q; q^2)_{\infty}q^k\sum_{\substack{l = \lceil\frac{2k}{3}\rceil \\ l \text{ even}}}^{k} q^{2\binom{k - l}{2} + k - l}\binom{k - l + \frac{3l - 2k - 2}{2}}{k - l}_{q^2}\sum_{m = 0}^{\infty}\frac{q^{2m(k - l + 1)}}{(-q^{2m + 1}; q^2)_{\frac{l}{2}}}\]
    \end{prop}

    \begin{proof}
    We begin by summing first the odd case of Proposition \ref{OddDistinctbyHook} and then the even case over all $m$ and $h$ and simplifying from there. For a fixed $k$ we have, 
    \[\sum_{\substack{m = 1 \\ m \text{ odd}}}^{\infty} \sum_{h = -\infty}^{k - 1} \sum_{\substack{l = \lceil\frac{2k-1}{3}\rceil \\ l \text{ odd}}}^{k} \frac{q^{(m - 1) (2 k - h - l) + k + l(k - h - 1) + 2 \binom{k - h}{2} + 2 \binom{k - l}{2}}(-q, q^2)_{\frac{m - 1}{2}}}{(q^2; q^2)_{k - h - 1}}\binom{k - l + \frac{3l - 2k}{2}}{k - l}_{q^2}\]
    Reordering the sums and reindexing $h \to - h + 1 - k$ gives
    \[q^k \sum_{\substack{l = \lceil\frac{2k-1}{3}\rceil \\ l \text{ odd}}}^{k}q^{2 \binom{k - l}{2}}\binom{k - l + \frac{3l - 2k}{2}}{k - l}_{q^2} \sum_{\substack{m = 1 \\ m \text{ odd}}}^{\infty}q^{(m - 1)(k - l + 1)}(-q; q^2)_{\frac{m - 1}{2}} \sum_{h = 0}^{\infty}\frac{q^{h^2 + h(l + m)}}{(q^2; q^2)_{h}}\]
    Using the identity (Corollary 2.2 in \cite{AndrewsBook})
    \[(-z; q)_{\infty}=\sum_{n=0}^{\infty} \frac{z^nq^{n(n-1)/2}}{(q; q)_n}\]
    the inner sum becomes precisely $(-q^{l+m+1}; q^2)_{\infty}$. Since the sum on $m$ is taken over odd values, we can reindex $m \to 2m + 1$ and take the sum over all $m$ to get.
    \[q^k \sum_{\substack{l = \lceil\frac{2k-1}{3}\rceil \\ l \text{ odd}}}^{k}q^{2 \binom{k - l}{2}}\binom{k - l + \frac{3l - 2k}{2}}{k - l}_{q^2} \sum_{m = 0}^{\infty}q^{2m(k - l + 1)}(-q; q^2)_m (-q^{l + 2m + 2}; q^2)_{\infty}\]
    Using the fact that $(-q; q^2)_m (-q^{l + 2m + 2}; q^2)_\infty = (-q; q^2)_\infty/(-q^{2m + 1}; q^2)_{\frac{l - 1}{2}}$ we get the first term of the Proposition. The case that $l$ is even follows from identical algebraic steps.
    
    \end{proof}

    \begin{rem}
    Asymptotic analysis on Proposition \ref{OddDistinctTotal} and Theorem 1.2 in \cite{AAOS} could settle Conjecture 5.3 in \cite{CDH}.
    \end{rem}

    \section{Proofs for Theorems in section 1}

    In this section we provide proofs for each of the theorems from section 1.
    
    \begin{proof}[Proof of Theorem 1.1]
        Setting $h = 0$, fixing $m$, and then summing Theorem \ref{mFixedbyHook} over all $k$ we get,
        \begin{align*}
            & \sum_{k = 1}^{\infty}\sum_{l = 1}^{k} \frac{q^{(m - 1) (2 k - l) + k + l(k - 1)}}{(q; q)_{m - 1}(q; q)_{k - 1}} \binom{k - 1}{l - 1}_q\\
            =& \frac{1}{(q; q)_{m - 1}}\sum_{l = 1}^{\infty}\sum_{k = l}^{\infty} \frac{q^{l(k - m) + k(2m - 1)}}{(q; q)_{l - 1}(q; q)_{k - l}}\\
            =& \frac{1}{(q; q)_{m - 1}}\sum_{l = 1}^{\infty}\sum_{k = 0}^{\infty} \frac{q^{l(k + l - m) + (k + l)(2m - 1)}}{(q; q)_{l - 1}(q; q)_{k}}\\
            =& \frac{1}{(q; q)_{m - 1}}\sum_{l = 1}^{\infty}\frac{q^{l(l + m - 1)}}{(q; q)_{l - 1}}\sum_{k = 0}^{\infty} \frac{q^{k(l + 2m - 1)}}{(q; q)_{k}}\\
            =& \frac{1}{(q; q)_{m - 1}}\sum_{l = 1}^{\infty}\frac{q^{l(l + m - 1)}}{(q; q)_{l - 1}}\cdot\frac{1}{(q^{l+2m-1}; q)_{\infty}}\\
            =& \frac{1}{(q; q)_{m - 1}(q; q)_{\infty}}\sum_{l = 1}^{\infty}q^{l(l + m - 1)}(q^{l}; q)_{2m - 1}\\
        \end{align*}
        The first line and the last line give the combinatorial descriptions in the theorem.
    \end{proof}

    \begin{proof}[Proof of Theorem 1.3]
    Set $k = m$ in Theorem \ref{FixedbyPart} to obtain,
    \[\sum_{s = 1 - h}^{\infty} \frac{q^{m(h - 1) + s(2m)}}{(q; q)_{m - 1}(q; q)_{s-1}}\binom{s + h - 1}{0}_q\]
    \[= \frac{q^{mh + m}}{(q; q)_{m - 1}}\sum_{s = -h}^{\infty} \frac{q^{2ms}}{(q; q)_{s}}\]
    \[= \frac{q^{mh + m}}{(q; q)_{m - 1}}\left(\frac{1}{(q^{2m}; q)_\infty} - \sum_{s = 0}^{-h - 1} \frac{q^{2ms}}{(q; q)_{s}}\right)\]
    Interpreting the first and last lines gives the theorem.
    \end{proof}

    Generalizing Theorem 3.4 from \cite{CHHK} we get the following theorem.

    \begin{proof}[Proof of Theorem 1.4]
    Using Theorem \ref{mFixedbyPart} we have
    \[\sum_{s = 0}^{\infty} \frac{q^{s(k + m) + k(k - h - m + 1)}}{(q;q)_{s + k - h - m} (q;q)_{m - 1}} \binom{s + k - m}{k - m}_q\]
    \[= \frac{q^{k(k - h - m + 1) - \binom{k - m + 1}{2}}}{(q;q)_{m - 1}}\sum_{s = 0}^{\infty} \frac{q^{s(k + m) + \binom{k - m + 1}{2}}}{(q;q)_{s + k - h - m}} \binom{s + k - m}{k - m}_q\]
    Interpreting the first and second lines gives the theorem.
    \end{proof}
    
    \begin{proof}[Proof of Theorem 1.5]
        For a fixed $m$ and $k$, summing Theorem \ref{mFixedbyHook} over all $h$ we have,
        \[\sum_{h = -\infty}^{k-1}\sum_{l = 1}^{k} \frac{q^{(m - 1) (2 k - h - l) + k + l(k - h - 1)}}{(q; q)_{k - h - 1}(q; q)_{m - 1}} \binom{k - 1}{l - 1}_q\]
        reindexing $h \to k - h - 1$ gives
        \[=\frac{q^{km}}{(q; q)_{m - 1}}\sum_{l = 1}^k q^{-(l - 1)(m - 1)}\binom{k - 1}{l - 1}_q \sum_{h = 0}^{\infty}\frac{q^{h(m + l - 1)}}{(q; q)_h}\]
        \[=\frac{q^{km}}{(q; q)_{m - 1}}\sum_{l = 1}^k q^{-(l - 1)(m - 1)}\frac{(q; q)_{k-1}}{(q; q)_{l-1}(q^{m + l - 1}; q)_\infty}\]
        the theorem follows by multiplying by $(q^m; q)_{l - 1}/(q^m; q)_{l - 1}$ and by the fact that
        \[\frac{(q; q)_{k-1}}{(q; q)_{m - 1}(q^m; q)_{l - 1}(q^{m + l - 1}; q)_\infty} = \frac{1}{(q^k; q)_\infty}\]
    \end{proof}


    \section{Future Work}

    There are many directions that future work could be taken. The same analysis required to produce generating functions for the sets of restricted partitions studied could be applied to more exotic families and potentially get similar generating functions. For example analogous results for self-conjugate partitions to mirror those of odd distinct partitions could be obtained. As mentioned in several remarks, asymptotic analysis on some of these theorems could potentially resolve conjectures posed by other authors on the relative distributions of hook lengths between different sets of equinumerous partitions. Another relatively straightforward extension would be to establish similar combinatorial results to those of the theorems in Section 1. All of these theorems use the generating functions established specifically for unrestricted partitions, but very similar analogous results could readily be established by looking at the generating functions that are regarding the hook lengths in any form of restricted partitions. A less straightforward, but nonetheless interesting, continuation would be to further delve into the connection that the author, Hemmer, Hopkins, and Keith in \cite{CHHK} saw with the truncated pentagonal number theorem of Andrews and Merca \cite{AndrewsMercaTruncatedPentagonal}. 
    
\bibliography{main}

\end{document}